\documentclass[reqno,10pt]{amsart}

\usepackage{latexsym,graphicx}
\usepackage{amsmath,amssymb}
\newtheorem{thm}{Theorem}[section]
\newtheorem{theorem}{Theorem}[section]
\newtheorem{cor}[thm]{Corollary}
\newtheorem{prop}[thm]{Proposition}

\newtheorem{fact}{Fact}

\newtheorem{quest}[thm]{Question}

\newtheorem{remark}{Remark}

\newtheorem{lemma}[thm]{Lemma}
\newtheorem{obs}{Observation}
\theoremstyle{definition}
\newtheorem{definition}{Definition}

\numberwithin{equation}{section}

\usepackage{xcolor}

\def\a{\alpha}

\def\epsilon{\varepsilon}
\def\eps{\varepsilon}
\def\e{\varepsilon}

\def\V{\mathcal{V}}
\def\dir{{\bf dir}}
\def\l{\ell}
\def\C{\mathbb{C}}

\def\ol{\overline}

\def\DB{{\rm DesBot}}

\def\cV{\mathcal{V}}

\begin{document}
\title[Excedance  set asymptotics]{Asymptotics of the extremal excedance  set statistic}
%\author[R. F. de Andrade, E. Lundberg, B. Nagle]{Rodrigo F. de Andrade, Erik Lundberg, and Brendan Nagle}

\author[R. F. de Andrade]{Rodrigo Ferraz de Andrade}
\address{Dept. of Mathematics, Purdue University, West Lafayette, IN 47906}
\email{rferrazd@math.purdue.edu}

\author[E. Lundberg]{Erik Lundberg}
\address{Dept. of Mathematical Sciences, Florida Atlantic University, Boca Raton, FL 33431}
\email{elundber@fau.edu}

\author[B. Nagle]{Brendan Nagle}\thanks{
The third author was partially supported by NSF Grant DMS~1001781.}  
\address{Dept. of Mathematics and Statistics, 
University of South Florida,  Tampa, FL
33620}
\email{bnagle@math.usf.edu}

\begin{abstract}
For non-negative integers $r$ and $s$ with $r + s = n-1$, let $[b^ra^s]$ denote the number of permutations
$\pi \in S_n$ which have the property that $\pi(i) > i$ if, and only if, $i \in [r] = \{1, \dots, r\}$.  
Answering a question of Clark and Ehrenborg (2010), we determine asymptotics on $[b^r a^s]$ when $r = \lfloor (n-1)/2 \rfloor$:
$$
\left[b^{\lfloor (n-1)/2 \rfloor} a^{\lceil (n-1)/2 \rceil}\right]
= 
\left( \frac{1}{2 \log 2\sqrt{  (1 - \log 2)}} + o(1) \right) \left( \frac{1}{ 2 \log 2 } \right)^n n!.  
$$
We also determine asymptotics on $[b^r a^s]$ for a suitably related $r = \Theta (s) \to \infty$.   
Our proof depends on multivariate asymptotic methods of R. Pemantle and M. C. Wilson.

We also consider two applications of our main result.  
One, we determine asymptotics on the number of permutations $\pi \in S_n$ which simultaneously avoid the vincular patterns 
21-34 and 34-21, i.e., for which $\pi$ is order-isomorphic to neither (2,1,3,4) 
nor $(3,4,2,1)$ on any coordinates $1 \leq i < i+1 < j < j+1 \leq n$.  
We also determine asymptotics on the number of $n$-cycles $\pi \in C_n$ which avoid stretching pairs, i.e., those for which $1 \leq \pi(i) < i < j < \pi(j) \leq n$.   \\

\end{abstract}

\maketitle

\section{Introduction}

The main results of this paper (Theorems \ref{thm:diag} and \ref{thm:multivar} below) consider the asymptotics of the number of permutations $\pi \in S_n$ with excedance word $b^ra^s$, where $r + s = n-1$.  
To present our results, we use the following notation and terminology, taken mostly from~\cite{ClEhr}.

\subsection{Preliminaries}  
Let $S_n$ denote the set of permutations on $[n] = \{1, \dots, n\}$.  
For $\pi \in S_n$ and $i \in [n-1]$, we say that $i$ is an {\it excedance} of $\pi$ if $\pi (i) > i$.  We write
$E(\pi) = \{i \in [n-1]: \pi (i) > i \}$ for the set of excedances of $\pi$.  
We define the {\it excedance word} $w(\pi) = (w_1,\dots, w_{n-1}) \in \{a, b\}^{n-1}$ of $\pi$ by, for each $i \in [n-1]$, 
$w_i = b$ if $i \in E(\pi)$ and $w_i = a$ if $i \not\in E(\pi)$, i.e., $\pi(i) \leq i$.    
For a word $w \in \{a, b\}^{n-1}$, let $[w]$ 
denote the number of permutations $\pi \in S_n$ that have excedance word $w$.  

In this paper, we consider words $w \in \{a, b\}^{n-1}$ of the following form.  
For non-negative 
integers 
$r + s = n-1$, let $b^r a^s \in \{a, b\}^{n-1}$ denote the word whose first $r$ coordinates equal $b$ and whose last $s$ coordinates equal $a$.  
R.~Ehrenborg and E.~Steingr\'{\i}msson~\cite{EhrSt}   
showed that $[b^ra^s]$ is maximized when $r = n - 1 - s \in \{ 
\lfloor (n-1)/2 \rfloor, \lceil (n-1)/2 \rceil \}$.  
E.~Clark and R.~Ehrenborg~\cite[Concluding Remarks]{ClEhr} posed the following problem.  
\begin{quest}\label{Q}
What are the asymptotics of 
 $[b^{\lfloor (n-1)/2 \rfloor} a^{\lceil (n-1)/2 \rceil}]$ as $n \rightarrow \infty$?
\end{quest}

In this paper, 
we address Question~\ref{Q} (see Theorem \ref{thm:diag} below), and more generally, in Theorem \ref{thm:multivar} we address the asymptotics of $[b^ra^s]$ for a wider range of $r + s = n-1$.  
We also provide some further applications of our main results.  

\subsection{Historical remark}

An important counterpart of the notion of an excedance is that of a \emph{descent}. 
For $\pi \in S_n$ and $i \in [n-1]$, we say that $i$ is a {\it descent} of $\pi$ if $\pi (i) > \pi(i+1)$.
Two of the classically studied permutation statistics are the number of descents and the number of excedances (both considered by MacMahon \cite{MacMahon}).
The more modern study of permutation statistics has included the dual consideration of 
the number of permutations with prescribed descent set (and, respectively, excedance set).
In particular, Niven \cite{Niven} and, indepedently, de Bruijn \cite{deBruijn}
showed that the most common descent set is realized by the \emph{alternating permutations}.
The alternating permutations are enumerated by the Euler numbers $E_n$ whose asymptotics are classical: 
\begin{equation}\label{eq:Andre}
 E_n = \left(\frac{4}{\pi} + o(1) \right) \left(\frac{2}{\pi}\right)^n n! 
\end{equation}
(This result follows from Andr\'e's Theorem \cite{Andre}.)
The most common excedance set was specified by the above-mentioned result \cite{EhrSt},
and the excedance set counterpart of equation (\ref{eq:Andre}) is provided by Theorem \ref{thm:diag} below.

\subsection{Main Results}  
Our first result provides an answer to Question~\ref{Q}.  

\begin{thm}\label{thm:diag}
%\begin{equation}\label{eq:asymptotics}
$$
\left[b^{\lfloor (n-1)/2 \rfloor} a^{\lceil (n-1)/2 \rceil}\right]
= 
\left( \frac{1}{2 \log 2\sqrt{  (1 - \log 2)}} + o(1) \right) \left( \frac{1}{ 2 \log 2 } \right)^n n!.  
$$
%\end{equation}
\end{thm}
\noindent We prove Theorem~\ref{thm:diag} in Section~\ref{sec:diag}.

More generally, we establish bivariate asymptotics for $[b^{r-1} a^s]$ when 
the parameters 
$r$ and $s$ tend to infinity 
within a certain {\it sector}.  (We define the concept of an {\it $\eps_0$-sector} $S_{\eps_0}$ in equation (\ref{eqn:sector}) below.)
Our result is somewhat technical to present, and we require several considerations.  
First, define the function $Q : \mathbb{R}^+ \times \mathbb{R}^+ \to \mathbb{R}$ by 
\begin{equation}
\label{eqn:functionQ}  
Q(x,y) = xy e^{-x-y} \left[ y e^{-y}+xe^{-x}-xy(e^{-y}+e^{-x}) \right], \quad x, y > 0.  
\end{equation} 
Second, 
define the function $f: \mathbb{R}^+ \to \mathbb{R}^+$ by 
\begin{equation}
\label{eqn:functionf}  
f(t) :=\frac{(1-e^t) \log(1-e^{-t})}{t}. 
%\quad \text{ \bf Fact: } f: \mathbb{R}^+ \to \mathbb{R}^+ \text{ is a strictly decreasing bijection.}
\end{equation} 
It is important for us that $f$ is invertible, as is claimed in the following 
fact\footnote{The proof of Fact~\ref{fact1} is standard, but for completeness, we have included a proof in the Appendix.}.      
\begin{fact}
\label{fact1}  
The function 
$f: \mathbb{R}^+ \to \mathbb{R}^+$ 
defined in~(\ref{eqn:functionf}) 
is a strictly decreasing bijection.
\end{fact}  
\noindent Third, 
set 
%we define the sector used in the statement of our bivariate asymptotic.
\begin{multline}
\label{eqn:sector}
\eps_0 = f(1) =
(e-1) (1- \log(e-1)) \approx 0.7881..., \\
\text{and define the {\it $\eps_0$-sector}}   
\quad   
S_{\eps_0} = \left\{ (r, s) \in \mathbb{N} \times \mathbb{N}: \eps_0 \leq \frac{s}{r} \leq \frac{1}{\eps_0} \right\}.   
\end{multline}  
%where we 
%call $S_{\eps_0}$ an {\it $\eps_0$-sector}.  
Our main result will provide
asymptotics for $[b^{r-1}a^s]$ for 
$r, s \to \infty$ with
$(r, s) \in S_{\eps_0}$.  
To that end, for each 
$(r, s) \in S_{\eps_0}$, 
set $x = x(r,s) := f^{-1}(s/r)$ and $y = y(r,s) := f^{-1}(r/s)$ (cf.~Fact~\ref{fact1}).  
Our main result is now stated as follows.   

\begin{thm}\label{thm:multivar}
For $(r, s) \in S_{\eps_0}$, let $x = x(r, s)$, $y = y(r, s)$, and $Q(x,y)$ be defined as above.  
Then, the following holds uniformly for $r, s \to \infty$ with $(r, s) \in S_{\eps_0}$:  
%\begin{equation}\label{eq:bivariate}
$$
 [b^{r-1} a^s] =  r! s! \left(e^{-y} + O(s^{-1/2}) \right) \frac{1}{\sqrt{2\pi}} x^{-r}y^{-s}\sqrt{\frac{y e^{-y}}{sQ(x,y)}}.  
$$
%\end{equation}
%where $$Q(x,y) 	= xy e^{-x-y} \left[ y e^{-y}+xe^{-x}-xy(e^{-y}+e^{-x}) \right] .$$
\end{thm}

\noindent  Theorem~\ref{thm:multivar} is proved in Section~\ref{sec:asymptotic} using 
a result of R.~Pemantle and M.C.~Wilson~\cite{PemWil2002, PemWil2008, PemWil2013}  
concerning multivariate asymptotics.

\subsection{Applications}  
We discuss two applications
of 
Theorem~\ref{thm:diag}.   
The first application falls within  
the well-studied area of 
{\it pattern avoidance}.  (Namely, we consider the simultaneous avoidance of a certain pair of vincular patterns). 
The second concerns 
objects related to vincular patterns, which have become known as {\it stretching pairs}.  
Initially, 
it will not be obvious how Theorem~\ref{thm:diag} provides the applications (see the two following corollaries) we state, but we will make this connection clear at the end of the Introduction.

\subsubsection{Vincular patterns}  
For $1 \leq k \leq n$, fix $\pi_0 \in S_k$ and $\pi \in S_n$.  We say that $\pi$ contains $\pi_0$ as a {\it pattern} 
if there exist $1\leq \l_1 < \dots < \l_k \leq n$ 
such that $(\pi(\l_1), \pi(\l_2), \dots , \pi(\l_k))$ is order-isomorphic to $\pi_0$.  
Now, write $\pi_0 = (a_1, \dots, a_k) = (\pi_0(1), \dots, \pi_0(k))$, and let 
$\pi_0^* = (a_1, \eps_1, a_2, \eps_2, \dots, \eps_{k-1}, a_k)$ be any sequence where, for each $1 \leq i \leq k$, $\eps_i$ is either a dash `$-$', or the empty string.  
We say that $\pi \in S_n$ admits $\pi_0^*$ as a {\it vincular pattern} 
(also known as a \emph{generalized pattern}\footnote{The notion was introduced by Babson and Steingrimsson \cite{BabsonSteingrimsson}.  
The terminology \emph{vincular pattern} first appeared in \cite{BCDK}.})
if $\pi_0$ occurs as a pattern at the positions 
$1 \leq \l_1 < \dots < \l_k \leq n$, and for all $1 \leq i \leq k$, 
$\eps_i \not= -$ implies $\l_{i+1} = \l_i + 1$, i.e., $\l_i, \l_{i+1}$ are consecutive in $\pi$.

As an illustrative example, we note that the permutation $(3,5,2,4,1) = 35241$ 
contains $132$ as a classical pattern (realized uniquely by the $3$, $5$, and $4$ occuring in that order).
However, $35241$ does not contain $1$-$32$ as a vincular pattern, since the $5$ and $4$ are not adjacent.

For $\pi_0 \in S_k$ and $\pi_0^*$ as above, write 
$\a_n(\pi_0^*)$ for the number of permutations $\pi \in S_n$ avoiding $\pi_0^*$ as a 
vincular pattern.  
Elizalde and Noy~\cite{EliNoy} 
studied the limiting behavior of $(\a_n(\pi_0^*)/n!)^{1/n}$ for several `dashless' (consecutive) vincular patterns $\pi_0^*$ of length 3.  For example, they showed that for 
$\pi_0^* = 123 = (1,2,3)$ (with no dashes), one has $(\a_n(123)/n!)^{1/n} \to 3 \sqrt{3}/(2\pi)$ as $n \to \infty$.  
Elizalde \cite{Eli} showed that for $\pi_0^* = \text{1-23-4} = (1, -, 2, 3, -, 4)$, one has $(\a_n (\text{1-23-4})/n!)^{1/n} \to 0$ as $n \to \infty$.  
%In~\cite{CLN2013}, it was shown that for all $\pi_0^*$, there exists $0 < d < 1$ so that $\a_n(\pi_0^*) < d^n n!$ holds for all sufficiently large integers $n$.  
In~\cite{CLN2013}, it was shown that for all sufficiently large even integers $n$, 
$(\a_n (\text{$21$-$34$})/n!)^{1/n} \geq 1/2 + o(1)$,   
which answered a question of S.~Elizalde~\cite{Eli}.  Using Theorem~\ref{thm:diag}, we are able to improve this last result by a factor of $1/\log 2$.  To that end, 
let $\a_n (\{\text{\rm 21-34}, \text{\rm 34-21}\})$ denote the number of permutations avoiding both \text{\rm 21-34} and \text{\rm 34-21} as vincular patterns.  

\begin{cor}
\label{cor:simult}
$$
\lim_{n\to \infty}  
\left(
\frac{
\a_n (\{\text{\rm 21-34}, \text{\rm 34-21}\})
}{n!} 
\right)^{1/n}  
= 
\frac{1}{2 \log 2} \approx 0.7213...
$$ 
As such, 
$$
\a_n (\text{$21$-$34$}) \geq 
\a_n (\{\text{\rm 21-34}, \text{\rm 34-21}\})
= 
\left(\frac{1}{2 \log 2} + o(1)  \right)^n n!.  
$$
\end{cor}  

\noindent We prove Corollary~\ref{cor:simult} at the end of this section.
It is an open problem to determine the exact value of the following limit (assuming that it exists):
$$\lim_{n\to \infty}  
\left(
\frac{
\a_n (\text{$21$-$34$})
}{n!} \right)^{1/n}.$$

\subsubsection{Stretching pairs}  
For a permutation $\pi \in S_n$, we call a pair $1\leq i < j \leq n$ a {\it stretching pair} if 
$\pi(i) < i < j < \pi (j)$.  
For example, in the permutation $\pi = (2, 1, 3, 5, 4) = 21354$, the pair $2 < 4$ is a stretching pair of $\pi$, since 
$\pi(2) = 1 < 2 < 4 < 5 = \pi(4)$.  
Stretching pairs
naturally arise in cyclic permutations 
within the context of 
a well-known result of Sharkovsky~\cite{Shark} in discrete dynamical systems.  More recently, 
stretching pairs in $n$-cycles were studied in~\cite{CLN2013, Lund, OEIS} from a combinatorial point of view.  
Let $C_n \subset S_n$ denote the set of $n$-cycles, and 
let $C_n^*$ denote those $n$-cycles which contain no stretching pairs.  
Theorem~\ref{thm:diag} allows us to 
determine the limiting behavior of $(|C_{n+1}^*|/n!)^{1/n}$.  

\begin{cor}\label{cor:limit}
$$ 
\lim_{n \rightarrow \infty} \left( \frac{|C^*_{n+1}|}{n!} \right)^{1/n} = \frac{1}{2\log 2}. 
$$ 
\end{cor}

\noindent We now proceed to the proofs of Corollaries~\ref{cor:simult} and~\ref{cor:limit}.

\subsection{Proofs of Corollaries~\ref{cor:simult} and~\ref{cor:limit}}  
The following lemma allows us to connect Theorem~\ref{thm:diag} 
to Corollaries~\ref{cor:simult} and~\ref{cor:limit}.  
As we show in Section~\ref{sec:connection}, 
Lemma~\ref{lem:connection}  
is (mostly) a consequence of a result 
of Clarke, Steingr\'{\i}msson, and Zeng~\cite[Proposition 3]{CSZ}.   

\begin{lemma}
\label{lem:connection}
$\,$  
\begin{enumerate}
\item
$|C_{n+1}^*| = \sum_{k=0}^{n-1} [b^k a^{n-1-k}]$.  
\item
$|C_{n+1}^*| 
\leq 
\a_n (\{\text{\rm 21-34}, \text{\rm 34-21}\})
\leq 
|C_{n+1}^*| 
+|C_{n+2}^*|
\leq 
2|C_{n+2}^*|$.  
\end{enumerate} 
\end{lemma}  

We now conclude Corollaries~\ref{cor:simult} and~\ref{cor:limit}.  
Indeed, recall that is was proven in~\cite{EhrSt}   
that $[b^ka^{n-1-k}]$ is maximized when $k \in \left\{ \lfloor (n-1)/2 \rfloor, \lceil (n-1)/2 \rceil \right\}$.  As such, we may use Statement~(1) of Lemma~\ref{lem:connection} to infer 
$$
\left[b^{\lfloor (n-1)/2 \rfloor} a^{\lceil (n-1)/2 \rceil}\right]
\leq 
|C_{n+1}^*| \leq 
n \left[b^{\lfloor (n-1)/2 \rfloor} a^{\lceil (n-1)/2 \rceil}\right]
$$
so that 
Corollary~\ref{cor:limit} is now immediate.  Using Corollary~\ref{cor:limit} and Statement~(2) of Lemma~\ref{lem:connection}, Corollary~\ref{cor:simult} is now immediate.

\begin{remark}
\rm
Statement~(1) of Lemma~\ref{lem:connection} 
also provides an exact combinatorial evaluation of $|C_{n+1}^*|$ that may be of independent interest.  
Indeed, Clark and Ehrenborg~\cite{ClEhr} 
showed that for all integers $M, N \geq 0$, 
$$
[b^Ma^N] = \sum_{i\geq 0} \big(S(M+1, i+1) S(N+1, i+1) i! (i+1)! \big), 
$$
where $S(s,t)$, $s, t \in \mathbb{N}$, denotes the Stirling number of the second kind.  
As such, 
$$
|C_{n+1}^*| = \sum_{k=0}^{n-1} 
\sum_{i\geq 0} \big(S(k+1, i+1) S(n-k, i+1) i! (i+1)! \big), 
$$
which complements our result in Corollary~\ref{cor:limit}.  
\end{remark}

\section{Proof of Lemma~\ref{lem:connection}}  
\label{sec:connection}  

First we state a Lemma from~\cite[Proposition 3]{CSZ}.
Define, 
for a permutation $\pi \in S_n$,
$\DB (\pi) = \{\pi(i): \pi(i) < \pi(i-1) \}$ to be the (so-called)
{\em Descent Bottoms Set}.

\begin{lemma}
[Clarke, Steingr\'{\i}msson, 
Zeng~\cite{CSZ}]     
\label{prop:bij}
There exists a bijection $\Phi :S_n \to S_n$ so that for any $\pi \in S_n$,
$E (\pi) = \DB (\Phi(\pi))$.
\end{lemma}

\noindent We include (for completeness) a proof of 
Lemma~\ref{prop:bij} using a different\footnote{According 
to Ehrenborg and Steingr\'{\i}msson~\cite{EhrSt},
there are several variations of the bijection $\Phi$,
originating with the early work of Foata and Sch\"utzenberger~\cite{FS}.}  
bijection $\Phi$ (which was suggested to us by Emeric Deutsch~\cite{ED}).

\begin{proof}[Proof of Lemma~\ref{prop:bij}]
Define the bijection $\Phi$ as follows.
Let $\pi \in S_n$ be given in its {\it standard cycle decomposition},
$SCD(\pi)$, 
i.e., each cycle begins with its smallest entry and cycles appear
in ascending order according to their initial entries.
Define $p = \Phi(\pi) \in S_n$ to be the permutation which, 
in 
one-line notation 
$p = (p_1, \dots, p_n)$ where $i \mapsto p_i$, 
is obtained by first reversing the order of entries
within each cycle of $\pi$, and then removing all parentheses.
For example, if $\pi = (1 \ 5 \ 2 \ 8)(3 \ 6 \ 7)(4 \ 9)$, then $\Phi(\pi) = (8,2,5,1,7,6,3,9,4)$.

It is easy to show that $\Phi$ is a bijection
by constructing its inverse function.  
Fix $p = (p_1, \dots, p_n) \in S_n$, 
and construct $SCD(\pi)$
for $\pi = \Phi^{-1}(p)$ as follows.
Find the index $k_1$ for which $p_{k_1}=1$.  Take the
consecutively indexed entries from $p_1$ to $p_{k_1}$ and
reverse their order to obtain the first cycle for $\pi$.
Find the index $k_2 > k_1$ for which $p_{k_2}$ is the smallest
element of $[n] \setminus \{p_1,\ldots,p_{k_1}\}$.  
Construct the second cycle for $\pi$ by reversing the order of the entries starting with $p_{k_1+1}$ and ending with $p_{k_2}$.  
Having constructed 
the first $j$ cycles for $\pi$, find the index $k_{j+1}>k_j$ so that $p_{k_{j+1}}$ is the smallest element of $[n] \setminus \{p_1,p_2, \ldots, p_{k_j} \}$, 
and obtain the next cycle by reversing the order of the entries starting with $p_{k_j+1}$ and ending with $p_{k_{j+1}}$.
By construction, $\pi$ will be in standard cycle decomposition and
$\Phi(\pi) = p$.

We now verify that $\Phi$ has the promised property:
for all $\pi \in S_n$, 
$E(\pi) = \DB (p = \Phi(\pi))$.    
To that end,
fix $\pi \in S_n$, and write $p = \Phi(\pi) = (p_1, \dots, p_n)$.  
For $j \in [n]$, we show $j \in E(\pi) \iff
j \in \DB (p)$.  We consider two cases,
depending on 
how $j$ appears in $SCD(\pi)$.  \\

\noindent {\bf Case 1.}  {\it In $SCD(\pi)$,
$j$ does not appear last in its cycle.}
Then, $\pi(j)$ is the next consecutive entry appearing in the same cycle
as $j$.  Set $k = p^{-1}(j)$ so that $p_k = j$.  Then, by the
construction of $\Phi$, $p_{k-1} = \pi(j)$.  Now,
$$
j \in E (\pi) \quad \iff \quad p_{k-1} = \pi(j) > j = p_k \quad \iff \quad
j = p_k \in \DB(p), 
$$
as desired.  

As an example of the argument above, consider again 
\begin{equation}
\label{repeatedexample}  
\pi = (1 \ 5 \ 2 \ 8)(3 \ 6 \ 7)(4 \ 9) \quad \text{and} \quad \Phi(\pi) = (8, 2, 5, 1, 7, 6, 3, 9, 4).  
\end{equation}  
If $j = 1$, then $\pi(1) = 5 > 1$, and so $1 \in E(\pi)$.  Correspondingly, $p(3) = p_3 = 5 > 1 = p_4 = p(4)$, and so 
$1 = p(4) \in \DB(p)$.  Alternatively, if $j = 5$, then $\pi(5) = 2 < 5$, and so $5 \not\in E(\pi)$.  
Correspondingly, $p(3) = p_3 = 5 > 2 = p_2 = p(2)$, and so $5 = p(3) \not\in \DB(p)$. \\

\noindent {\bf Case 2.}  {\it  In $SCD(\pi)$,
$j$ appears last in its cycle.}
Then, $\pi(j)$ is the first entry in the cycle containing $j$.  Since $\pi$
is in standard cycle decomposition, $\pi(j) \leq j$, in which case $j\not\in
E(\pi)$.  If $j$ is in the first cycle of $\pi$, then
$p_1 = j \not\in \DB(p)$, as desired.  Assume, therefore, that
$j$ is not in the first cycle of $\pi$.  Let the cycle immediately preceding
that of 
$j$ begin with $a$, and let the cycle containing $j$ begin with $b = \pi(j)$.
Then, $a < b = \pi(j) \leq j$.  Moreover, in the notation of Case~1, we have $p_{k-1} = a < j = p_k$ (for some $k \in [n]$)  so that
$p_k = j \not\in \DB(p)$, as desired.

As an example of the argument above, consider again the example 
of $\pi$ and $\Phi(\pi)$ in~(\ref{repeatedexample}).    
If $j = 7$, then $\pi(7) = b = 3 < 7$, and so $7 \not\in E(\pi)$.  Correspondingly, $a = 1 = p_4 = p(4)
< 7 = p_5 = p(5)$, and so $7 = p(5) \not\in \DB(p)$.  
\end{proof}

\subsection{Proof of Lemma~\ref{lem:connection}: Statement~(1)}  
To prove Statement~(1) of Lemma~\ref{lem:connection} from Lemma~\ref{prop:bij}, we make the following two observations.  

\begin{obs}\label{first}
Let $\pi \in S_{n+1}$ have no fixed points.  
Then,
$\pi$ admits no stretching pairs
if and only if there exists $\ell \in [n]$ so that
$E (\pi) = [\ell]$.  
\end{obs}

\begin{proof} Indeed, clearly no $\pi \in S_{n+1}$ whose excedence set takes the form $E (\pi) = [\ell]$ can have stretching pairs.  
Suppose $\pi \in S_{n+1}$ has neither stretching pairs nor fixed points, and let 
$\ell \in [n]$ be the maximum integer for which $[\ell] \subseteq E(\pi)$.  
Then $[\ell] = E (\pi)$, since otherwise $j \in E (\pi) \setminus [\ell]$ would give the stretching pair $\ell+1 < j$.  
\end{proof}  

\begin{obs}
\label{second}
There is a bijection $F: C_{n+1} \to S_{n}$
so that for each
$\pi \in C_{n+1}$ and for each $\l \in [n]$, 
$E (\pi) = [\ell]$ if and only if
$\DB (F(\pi)) = [\ell -1]$.  
%and each
%$j \in \{1, \dots, n-1\}$,
%we have $j+1 \in \Exc (\pi)$ iff $j \in \DB(F(\pi))$.
\end{obs}

\begin{proof}
Indeed, fix $\pi \in C_{n+1}$ and define $F(\pi) = p \in S_n$ by the rule $p(i) = \pi^{n+1-i}(1) - 1$.  For $j \in [n]$, we show $j+1 \in E(\pi)$ if and only if $j \in \DB(p)$.  
Write $j = p(i)$ for some $i \in [n]$ so that 
$j + 1 = p(i) + 1 = \pi^{n+1-i}(1)$ and 
$\pi (j+1)  = \pi^{n+1 - (i-1)}(1) = p(i-1) + 1$.  
Now, $\pi(j+1) - (j+1) = p(i-1) - p(i)$ is positive if, and only if, $j+1 \in E(\pi)$ and if, and only if, $j = p(i) \in \DB(p)$.  
\end{proof}

To prove Statement~(1) of Lemma~\ref{lem:connection}, we note that every $\pi \in C_{n+1}$ is without fixed points, and so Observation~\ref{first} gives the disjoint union
$$
C_{n+1}^* = \bigcup_{\ell=1}^n 
\big\{\pi \in C_{n+1}: E (\pi) = [\l] \big\}  \quad \implies \quad 
\left|C_{n+1}^*\right| = \sum_{\ell=1}^n 
\big|\left\{\pi \in C_{n+1}: E (\pi) = [\l] \right\}\big|.  
$$
Applying Observation~\ref{second} and Lemma~\ref{prop:bij} (in this order) then gives 
$$
\left|C_{n+1}^*\right| = 
\sum_{k=0}^{n-1} \big| \left\{\pi \in S_n: \DB (\pi) = [k] \right\} \big|  
= 
\sum_{k=0}^{n-1} \big| \left\{\pi \in S_n: E (\pi) = [k] \right\} \big|  
= 
\sum_{k=0}^{n-1} 
[b^k a^{n-1-k}], 
$$
as promised.

\subsection{Proof of Lemma~\ref{lem:connection}: Statement~(2)}  
Let $\a_n = \a_n(\{\text{21$-$34}, \text{34$-$21}\})$
denote the number of permutations $\pi \in S_n$ 
which avoid both of the vincular patterns 21$-$34 and 34$-$21.
For an $(n+1)$-cycle $\pi \in C_{n+1}$, we say that a stretching pair $\pi(i) < i < j < \pi (j)$ is {\it typical} if $\pi (j) \neq n+1$, and {\it exceptional} otherwise.  
Let $E_{n+1} \subset C_{n+1}$ denote the set of 
$(n+1)$-cycles $\pi$ whose only stretching pairs are exceptional.  
The following statement\footnote{The bijection $\phi$ in Proposition~\ref{prop:Cooperetal} is easy to state.  If 
$(n+1 \ a_1 \ \dots \ a_n) \in C_{n+1}$ is given in cyclic notation
$n+1 \mapsto a_1 \mapsto \dots \mapsto a_n \mapsto n+1$, 
then $\phi(\pi) = (a_1, \dots, a_n) \in S_n$, which we've written in customary notation $i \mapsto a_i$.}  
is an easy fact from Section~4.1 of~\cite{CLN2013}.

\begin{prop}[Cooper et al.~\cite{CLN2013}]
\label{prop:Cooperetal}  
There is a bijection $\phi: C_{n+1} \to S_n$ 
with the property that 
$\pi \in C^*_{n+1} \cup E_{n+1}$ if and only if 
$\phi(\pi)$ avoids both of the vincular patterns 21$-$34 and 34$-$21.  
\end{prop}  

Proposition~\ref{prop:Cooperetal} gives the identity $\a_n = |E_{n+1}| + |C^*_{n+1}|$, and hence, $\a_n \geq |C^*_{n+1}|$, which is the lower bound in Statement~(2) of Lemma~\ref{lem:connection}.  
It remains to prove the upper bound.  
We shall establish an injection $\iota : E_{n+1} \to C^*_{n+2}$, in which case $|E_{n+1}| \leq |C^*_{n+2}|$ and the upper bound in Statement~(2) of Lemma~\ref{lem:connection} then follows.

To define $\iota$, it will be convenient to work rather with $C_{n+2}[0, n+1]$, which is the set of $(n+2)$-cycles defined on the elements $[0, n+1] = \{0, 1, \dots, n+1\}$.
We take $C^*_{n+2}[0, n+1]$ to be the set of $\pi \in C_{n+2}[0, n+1]$ with no stretching pairs.  
Now, fix $\pi \in E_{n+1}$, and write $j_0 = \pi^{-1}(n+1)$.  By construction, all stretching pairs $i < j$ in $\pi \in E_{n+1}$ have $j = j_0$.  
Define $p = \iota(\pi) \in C^*_{n+2}[0, n+1]$ by the following rule:  $p(0) = n+1$, $p(j_0) = 0$, and $p(i) = \pi(i)$ for all remaining $i \in [n+1]$.  Clearly, $\iota$ is an injection 
and 
$p = \iota(\pi)$ is an $(n+2)$-cycle on the elements $[0, n+1]$.
Note that 
all stretching pairs of $\pi$ have been eliminated in $p = \iota (\pi)$, and no new ones arise.

%%%%%%%%%%%%%%%%%%%%%%%%%%%%%%%%%%%%%%%%%%%%%%%%%%%%%%%%%%%%%%%%%%%%%%%%%%
%%%%%%%%%%%%%%%%%%%%%%%%%%%%%%%%%%%%%%%%%%%%%%%%%%%%%%%%%%%%%%%%%%%%%%%%%%
%%%%%%%%%%%%%%%%%%%%%%%%%                      %%%%%%%%%%%%%%%%%%%%%%%%%%%
%%%%%%%%%%%%%%%%%%%%%%%%% Proof of Main Result %%%%%%%%%%%%%%%%%%%%%%%%%%%
%%%%%%%%%%%%%%%%%%%%%%%%%                      %%%%%%%%%%%%%%%%%%%%%%%%%%%
%%%%%%%%%%%%%%%%%%%%%%%%%%%%%%%%%%%%%%%%%%%%%%%%%%%%%%%%%%%%%%%%%%%%%%%%%%
%%%%%%%%%%%%%%%%%%%%%%%%%%%%%%%%%%%%%%%%%%%%%%%%%%%%%%%%%%%%%%%%%%%%%%%%%%

\section{Proof of Theorem~\ref{thm:multivar}}
\label{sec:asymptotic}

We begin by outlining the main approach to our proof of Theorem~\ref{thm:multivar}.  To that end, we note a result of 
Clark and Ehrenborg
(see \cite[Thm. 3.1]{ClEhr}) that $[b^ra^s]$ has 
bivariate exponential generating function 
\begin{equation}
\label{eq:BEGF}
\sum_{r,s \geq 0} \frac{[b^r a^s]}{r! s!} x^r y^s = \frac{e^{-x} e^{-y}}{(e^{-x} + e^{-y} - 1)^2}
= 
\frac{\partial}{\partial x} \left( \frac{e^{-y}}{e^{-x} + e^{-y} - 1} \right).
\end{equation}
Now, 
for $(x, y)$ in a neighborhood of $(0,0)$, 
write 
\begin{equation}
\label{eqn:goal}  
\frac{e^{-y}}{e^{-x} + e^{-y} - 1} = \sum_{r, s \geq 0} A_{r,s} x^r y^s.
\end{equation}  
We will apply 
a result of Pemantle and Wilson~\cite{PemWil2002, PemWil2008, PemWil2013}
to the coefficients $A_{r,s}$ when $(r, s) \in S_{\eps_0}$
(cf. equation (\ref{eqn:sector})), which will provide 
asymptotics on these coefficients.  
From there, 
Theorem~\ref{thm:multivar} will follow, since by 
equation (\ref{eq:BEGF}) 
we have 
\begin{equation}
\label{eqn:finalline}
[b^{r-1}a^s] = r! s! A_{r,s}.  
\end{equation}  

While the plan above is straightforward, the details take some work.  Indeed, the result of Pemantle and Wilson
is a highly technical statement, and much of our work will be in showing that it can be applied to the setting we need.
Let us now proceed to the result of Pemantle and Wilson.  

\subsection{Preliminaries and the result of Pemantle and Wilson}   
In all that follows, 
suppose $F: \mathbb{C}^2 \to \mathbb{C}$ is a meromorphic function, where we write $F(x,y) = G(x,y)/H(x,y)$ for some holomorphic functions $G, H : \mathbb{C}^2 \to \mathbb{C}$.  
We
write 
$\cV = \cV_F = \left\{(x, y) \in \mathbb{C}^2: H(x, y) = 0\right\}$    
for the variety of singularities of $F$.  
We say that $H$ {\it vanishes to order one} on $\cV$ if $\nabla H (x, y) \not= \vec{0}$ for each $(x, y) \in \cV$.  
We write ${\bf dir}(x, y) = {\rm span}_{\mathbb{C}} \left\{(x H_x, y H_y)\right\}$.

On the variety $\cV$, we have the following important 
concept of a {\it strictly minimal point}.  

\begin{definition}\label{def:sm}
A point $(x,y) \in \V$ is called \emph{strictly minimal} if the closed bidisk 
$$\ol{D}(0,|x|) \times \ol{D}(0,|y|) = \{ (z,w) \in \C^2 : |z| \leq |x|, |w| \leq |y| \}$$
intersects $\V$ only at the point $(x,y)$.
\end{definition}

The following result of Pemantle and Wilson appeared as Theorem~3.1 in~\cite{PemWil2002}, as Corollary~3.21 in~\cite{PemWil2008}, and as 
Theorem~9.5.7 in~\cite{PemWil2013}.  

\begin{theorem}[Pemantle, Wilson~\cite{PemWil2002, PemWil2008, PemWil2013}]  
\label{thm:PW}
Let $F = G/H : \mathbb{C}^2 \to \mathbb{C}$ be a meromorphic function
with variety of singularities $\cV = \cV_F$.  
Write $F(x, y) = \sum_{r, s \geq 0} A_{r,s} x^r y^s$ outside of $\cV$, and assume $H$ vanishes to order one on $\cV$.  

Fix $\eps > 0$, 
and 
assume that 
for each $(r,s) \in S_{\eps}$ (cf. equation (\ref{eqn:sector})), 
there 
exists 
$(x, y) \in \cV$ so that the following conditions hold:  
\begin{enumerate}  
\item[$(i)$]  
$(r, s) \in \dir(x, y)$, where 
the point $(x, y) \in \cV$ is strictly minimal; 
\item[$(ii)$]  
As $(r, s) \in S_{\eps}$ varies, the point 
$(x, y) \in \cV$ varies smoothly over some compact set; 
\item[$(iii)$]  
The point $(x, y) \in \cV$ satisfies that 
$G(x,y) \not= 0$, 
and also, 
$$
Q(x,y):= -y^2H_y^2xH_x-yH_yx^2H_x^2-x^2y^2(H_y^2H_{xx}+H_x^2H_{yy}-2H_xH_yH_{xy}) \not= 0.  
$$
\end{enumerate}  
Then as 
$r, s \to \infty$ with $(r, s) \in S_{\eps}$, we have 
\begin{equation}\label{eq:PW}
 A_{r,s} = \left( G(x,y) + O(s^{-1/2}) \right) \frac{1}{\sqrt{2\pi}}x^{-r}y^{-s}\sqrt{\frac{-yH_y}{sQ(x,y)}},
\end{equation}
where the error estimate $O(s^{-1/2})$ is uniform over $S_\e$. 
\end{theorem}

\subsection{Deriving Theorem~\ref{thm:multivar} from Theorem~\ref{thm:PW}}  
\label{sec:PWimpliesmultivar}  
From equation (\ref{eq:BEGF}), set 
\begin{multline}
\label{eqn:setup}  
F(x,y) = \frac{e^{-y}}{e^{-x} + e^{-y} - 1}, \quad 
G(x,y) = e^{-y}, \quad 
H(x, y) = e^{-x} + e^{-y} - 1, \\
\text{and} \quad  \cV = \cV_F = \left\{(x, y) \in \mathbb{C}^2: H(x, y) = e^{-x} + e^{-y} - 1 = 0 \right\}.  
\end{multline}  
Clearly, $G$ and $H$ are holomorphic functions so that $F$ is meromorphic with variety of singularities given by $\cV$.  
Clearly, $\nabla H$ is never zero, and therefore $H$ vanishes to order one on $\cV$.

Fix $(r, s) \in S_{\eps_0}$ (cf. equation (\ref{eqn:sector})).   
To apply Theorem~\ref{thm:PW}, we must guarantee the existence of a point $(x, y) \in \cV$ so that Conditions~$(i)$--$(iii)$ 
above hold.  
To find the desired point $(x,y) \in \cV$, consider the equation
\begin{equation}
\label{eqn:findxy}
r \frac{e^x}{x} = s \frac{e^y}{y} \quad \text{on the variety $\cV$}.
\end{equation}  
Using the identity $y = - \log (1 - e^{-x})$ on $\cV$, we see that equation (\ref{eqn:findxy}) holds if, 
and only if, 
$$
y = \frac{s}{r} \cdot \frac{xe^y}{e^x} = \frac{s}{r} \cdot \frac{x}{e^x - 1}.  
$$
Again using $y = - \log (1 - e^{-x})$ on $\cV$, we see that equation (\ref{eqn:findxy}) holds for $x$ satisfying 
\begin{equation}
\label{eqn:findx}  
- \log \left(1 - e^{-x} \right) = \frac{s}{r} \cdot \frac{x}{e^x - 1} \quad \implies \quad \frac{s}{r} = \frac{(1 - e^x) \log (1 - e^{-x} )}{x}.  
\end{equation} 
Recall from equation (\ref{eqn:functionf}) that $f$ is invertible.  

Returning to equation (\ref{eqn:findx}), 
one solution $(x, y) \in \cV$ to equation (\ref{eqn:findxy}) has $x = f^{-1} (s/r) > 0$, 
where $f$ is the function defined by equation (\ref{eqn:functionf}).  
Using symmetric calculations and $x = - \log (1 - e^{-y})$ on $\cV$, we also find that this solution 
has $y = f^{-1}(r/s) > 0$.  

Thus, for fixed $(r, s) \in S_{\eps_0}$, we have identified the promised point $(x, y) \in \cV$ by $x = f^{-1}(s/r)$ and $y = f^{-1}(r/s)$.  
It remains to show that this fixed
point $(x, y) \in \cV$ satisfies Conditions~$(i)$--$(iii)$ of Theorem~\ref{thm:PW}.

\subsubsection*{Verifying Condition~$(i)$}  
Consider 
${\bf dir}(x, y) = {\rm span}_{\mathbb{C}} \left\{ (xH_x, yH_y) \right\}  
\stackrel{\text{(\ref{eqn:setup})}}{=} {\rm span}_{\mathbb{C}} \left\{ (-xe^{-x}, -y e^{-y}  ) \right\}$.  
Since $x = f^{-1}(s/r)$ and $y = f^{-1}(r/s)$ are solutions of equation (\ref{eqn:findxy}), write 
$$
k = r\frac{e^x}{x} = s \frac{e^y}{y} \quad \implies \quad r = k x e^{-x} \text{ and } s = k y e^{-y}  
\quad \implies \quad (r, s) \in {\bf dir}(x, y), 
$$
as promised.  
To see 
that $(x, y)$ is strictly minimal, we will apply the following lemma (which we prove in the next subsection).  

\begin{lemma}\label{lemma:strmin}
If $(a, b) \in \cV \cap (0, 1)^2$
(cf. equation (\ref{eqn:setup})),   
then 
$(a, b)$ is strictly minimal on $\cV$.  
\end{lemma}

\begin{proof}[Proof of Lemma \ref{lemma:strmin}]
Fix $(a, b) \in \cV \cap (0, 1)^2$, where we recall 
from equation (\ref{eqn:setup})    
that 
$$
\V = 
\left\{(x, y) \in \mathbb{C}^2: H(x, y) = e^{-x} + e^{-y} - 1 = 0 \right\}.     
$$
To show that $(a, b)$ is stictly minimal on $\cV$, 
we prove that 
\begin{equation}
\label{eqn:strmin1}
\min_{z \in \ol{D}(0, a)} \Re e \, e^{-z}
\text{ 
is achieved at only $z = a$}, 
\end{equation}  
and so symmetrically, 
$$
\min_{w \in \ol{D}(0, b)} \Re e \, e^{-w}
\text{ 
is achieved at only $w = b$}.  
$$
As such, if $(z, w) \in \ol{D}(0, a) \times \ol{D}(0, b)$ satisfies $(z, w) \neq (a, b)$, we have 
$$
\Re e \, H(z, w) = \Re e \, e^{-z} + \Re e \, e^{-w} - 1 > 
e^{-a} + e^{-b} - 1  = 0,   
$$
in which case $H(z, w) \neq 0$ and so $(z, w) \not\in \cV$.

To show the statement (\ref{eqn:strmin1}), we 
solve for all $z \in \ol{D}(0, a)$ which 
minimize 
$\Re e \, e^{-z}$ (and show that only $z = a$ works).  
We 
begin by making the following initial considerations.  
For $z \in \ol{D}(0, a)$, we write 
$z = u + i v$ so that 
$u^2 + v^2 \leq a^2$ and 
$\Re e \, e^{-z} = e^{-u} \cos v$.  Note that $\Re e \, e^{-z} = e^{-u} \cos v$ is a harmonic function (it is the real part of an analytic function).  As such, the maximum principle \cite[Sec. 6.2]{Ahlfors} ensures 
that the minimum value of $\Re e \, e^{-z} = e^{-u} \cos v$
over $\ol{D}(0, a)$ is attained at the boundary $\partial D (0, a) = \{z \in \ol{D}(0, a): |z|^2 = u^2 + v^2 = a^2\}$.  
As such, we 
will always assume $v \neq 0$, for otherwise
$z = u = \pm a$, and 
the following hold:  
\begin{enumerate}
\item
$z = -a$ does not minimize 
$\Re e \, e^{-z} = e^{-u} \cos v$ 
(since $e^a > e^{-a}$ with $a > 0$), 
\item
$z = a$ is what we
promise in the statement (\ref{eqn:strmin1}).
\end{enumerate}  

To minimize 
$\Re e \, e^{-z} = e^{-u} \cos v$
subject to $|z|^2 = u^2 + v^2 = a^2$, 
we use the Lagrange multiplier rule:
for a scalar $\lambda \in \mathbb{R}$, set 
\begin{equation}
\label{eqn:Lagrange}
\nabla \left(e^{-u} \cos v\right) = \lambda \nabla \left(u^2 + v^2\right) 
\quad \implies \quad 
\left(- e^{-u} \cos v, - e^{-u} \sin v\right) = \left(2 \lambda u, 2 \lambda v\right).  
\end{equation}  
In light of the rule (\ref{eqn:Lagrange}), 
we may now also assume $u \neq 0$, for otherwise, $\cos v = 0$ implies $0 \neq v = k \pi/ 2$ for some (necessarily nonzero) $k \in \mathbb{Z}$, and so 
$1 > a = v^2 = k^2 \pi^2/4 \geq \pi^2/ 4> 2$.  
At this stage of our analysis, we will have 
proven the statement (\ref{eqn:strmin1}) if we can show 
that 
\begin{equation}
\label{eqn:strmin2}
\text{no $z = u + i v \in \partial \ol{D}(0, a)$ with $u \neq 0 \neq v$ satisfy equation (\ref{eqn:Lagrange}).}  
\end{equation}

Indeed, 
with $u \neq 0 \neq v$, we may 
rewrite equation (\ref{eqn:Lagrange}) to say 
\begin{equation}
\label{eqn:polar0}  
\frac{\sin v}{v} = -2 \lambda e^{-u} = \frac{\cos v}{u} \quad \implies \quad v = u \tan v, 
\end{equation}  
which we now rewrite in polar coordinates.  For $\theta \in [0, 2\pi)$, let $u = a \cos \theta$ and $v = a \sin \theta$, where we may assume from $u \neq 0 \neq v$ that $\theta \not\in \{0, \pi/2, \pi, 3\pi / 2\}$.  
Then equation (\ref{eqn:polar0}) is equivalent to 
\begin{equation}
\label{eqn:polar}
a \sin \theta = a \cos \theta \tan (a \sin \theta) \quad \implies \quad \tan \theta = \tan (a \sin \theta) \quad \implies \quad a \sin \theta = \theta + k \pi \text{ where $k \in \mathbb{Z}$}.  
\end{equation}
We now investigate the possible values of $k \in \mathbb{Z}$ in equation (\ref{eqn:polar}).  Since $a \in (0, 1)$
and $\theta \in (0, 2\pi)$, we have 
\begin{equation}
\label{eqn:trig}
|\theta + k \pi| = a |\sin \theta|  
< |\sin \theta| 
< 
\min \{1, \theta\}  
\end{equation}  
and so only $k \in \{-1, -2\}$ are possible.   
However, $k \neq - 1$ because the two sides of $a \sin \theta = \theta - \pi$ will have opposite signs.   
Moreover, $k \neq -2$ because
$$
2\pi - \theta = -a \sin \theta = a \sin (-\theta) = a \sin (2\pi - \theta) \leq a |\sin (2\pi - \theta)| < |\sin (2\pi - \theta)| \leq 2\pi - \theta.
$$
Thus, we have proven the statement (\ref{eqn:strmin2}), and hence
the statement (\ref{eqn:strmin1}), which concludes the 
proof of Lemma \ref{lemma:strmin}.  
\end{proof}

To apply Lemma~\ref{lemma:strmin}, we have already noted that $x, y > 0$ (cf.~Fact~1), so it remains to show that $x, y < 1$.  
Now, since 
$(r, s) \in S_{\eps_0}$ (cf. equation (\ref{eqn:sector})), we have $\tfrac{s}{r}, \tfrac{r}{s} > \eps_0$, where we note from equation (\ref{eqn:sector}) 
that $f(1) = \eps_0$.  Thus, since $f$ is monotone decreasing, we must have $x = f^{-1}(s/r) < 1$ and 
$y = f^{-1}(r/s) < 1$ so that, by Lemma~\ref{lemma:strmin}, $(x, y)$ is strictly minimal.

\begin{figure}[t]
    \begin{center}
    \includegraphics[scale=.27]{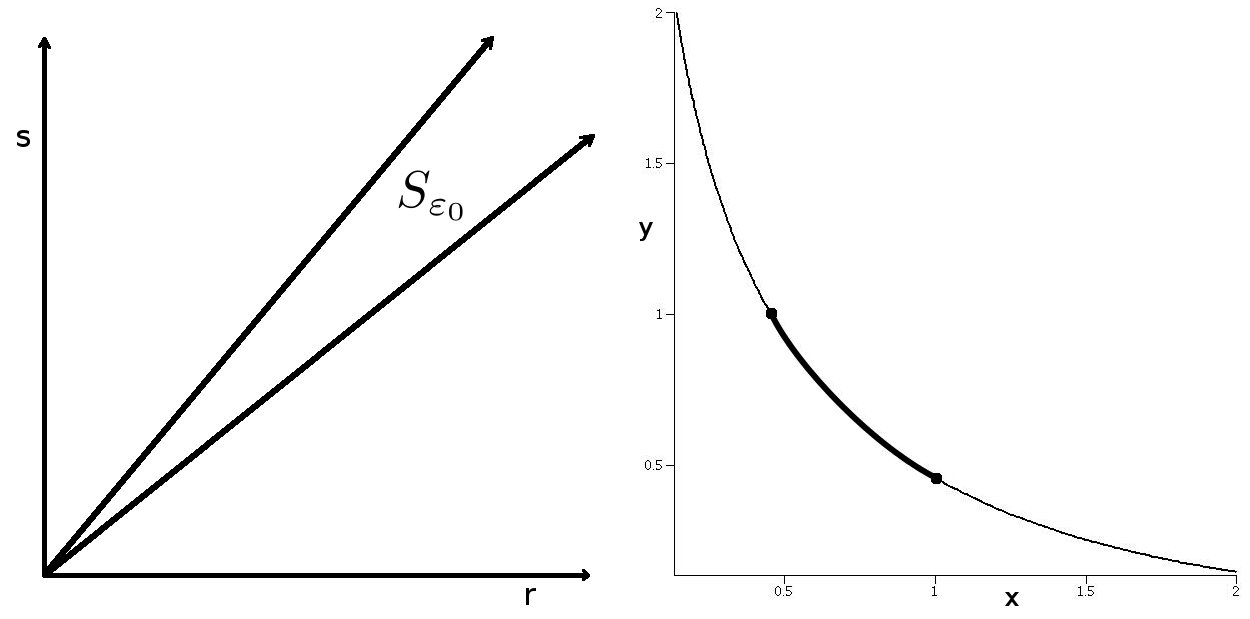}
    \end{center}
    \caption{The Sector $S_{\e_0}$ in the $rs$-plane (left), and the real section of $\cV$ (right) in the $xy$-plane with the set of minimal points $(x,y)$ corresponding to $S_{\e_0}$ plotted in bold.}
    \label{fig:Both}
\end{figure}

\subsubsection*{Verifying Condition~$(ii)$}  
It is easy to verify Condition~$(ii)$.  Indeed, by the Inverse Function Theorem, 
$(x, y) \in \cV$ varies smoothly as $(r, s) \in S_{\eps_0}$ varies.  Moreover, $(x, y) \in \cV$ varies over
the trace of the curve $(f^{-1}(t),f^{-1}(1/t))$ for $t \in [\eps_0, 1/\eps_0]$ (see Figure \ref{fig:Both}), which is the continuous image of a compact set.

\subsubsection*{Verifying Condition~$(iii)$}  
It is slightly tedious to verify Condition~$(iii)$.  
To begin, note that the function $G(x,y) = e^{-y}$ 
from equation (\ref{eqn:setup}) is never zero.
We show that the function $Q(x,y)$ is always positive on $\cV \cap (\mathbb{R}^+ \times \mathbb{R}^+)$.  
It is easy to check that 
\begin{multline*}  
Q(x,y) = -y^2H_y^2xH_x-yH_yx^2H_x^2-x^2y^2(H_y^2H_{xx}+H_x^2H_{yy}-2H_xH_yH_{xy})  \\
= 
xy e^{-x-y} \left[ y e^{-y}+xe^{-x}-xy(e^{-y}+e^{-x}) \right], 
\end{multline*}  
which we defined 
in equation (\ref{eqn:functionQ}).    
Now, 
$x  > 0$ and 
$y > 0$ 
imply that $xye^{-x-y} > 0$, so we will disregard this factor.  Moreover, in the expression above, $e^{-y} + e^{-x} = 1$ on $\cV$, so we consider
the simpler function
$$
P(x,y) = ye^{-y} + xe^{-x} - xy, 
$$
and prove that $P(x,y)$ is positive on $\cV \cap (\mathbb{R}^+ \times \mathbb{R}^+)$.  
Using $1 + x < e^x$ on $\cV$, we have 
\begin{multline*}  
1 + x < e^x = \frac{1}{1 - e^{-y}} \quad \implies \quad x < \frac{e^{-y}}{1-e^{-y}} \quad \implies \\
P(x, y) = 
ye^{-y} - x \left(y - e^{-x}\right) 
> 
ye^{-y} -  
\left(\frac{e^{-y}}{1-e^{-y}}\right)   
\left(y - e^{-x}\right) 
= 
\frac{e^{-y}}{1-e^{-y}}
\left(e^{-x} - ye^{-y} \right).  
\end{multline*}  
Clearly, the first factor above is positive on $y > 0$, so it suffices to consider $R(x,y) = e^{-x} - y e^{-y}$.  However, on $\cV$, we have 
$R(x,y) =1 - e^{-y} (1 + y) > 0$, 
since $1 + y < e^y$.  
This confirms Condition~$(iii)$.  \\

We now conclude the proof of Theorem~\ref{thm:multivar}.  
We apply Theorem~\ref{thm:PW} to 
$e^{-y}/(e^{-x} + e^{-y} - 1)$ 
(cf. equations (\ref{eqn:goal}) and (\ref{eqn:setup})) to conclude that for $r, s \to \infty$ with $(r, s) \in S_{\eps_0}$, 
$$
A_{r,s} = \left(e^{-y} + O(s^{-1/2}) \right) \frac{1}{\sqrt{2\pi}} x^{-r} y^{-s} \sqrt{\frac{ye^y}{sQ(x,y)}}  
$$  
and so, by equation (\ref{eqn:finalline}), we have 
$$
[b^{r-1}a^s] = 
r! s! A_{r,s} 
=
r! s! \left(e^{-y} + O(s^{-1/2}) \right) \frac{1}{\sqrt{2\pi}} x^{-r} y^{-s} \sqrt{\frac{ye^y}{sQ(x,y)}},   
$$
as promised.

\section{Proof of Theorem~\ref{thm:diag}}  
\label{sec:diag}  

The proof of Theorem~\ref{thm:diag} involves specializing Theorem~\ref{thm:multivar} 
to the case $r - 1 = \lfloor (n-1)/2 \rfloor$ and $s = \lceil (n-1)/2 \rceil$, and therefore, consists of calculations.  
We split the proof into two cases, depending on the parity of $n$, and begin with the easier (cleaner) case.  

\subsection{Proof of Theorem~\ref{thm:diag} ($n$ is even)}  
When $n$ is even, we have $r = s = n/2$, and so 
$x = y = f^{-1}(1)$, where $f$ is the function defined in 
equation (\ref{eqn:functionf}).
As we saw in the proof of Theorem~\ref{thm:multivar}, $(x, y)$ is a strictly minimal point on~$\cV = \{(x, y) \in \mathbb{C}^2: e^{-x} + e^{-y} - 1 = 0\}$, and so 
we easily calculate $x = y = \log 2$.  
It remains to substitute $r = s = n/2$ and $x = y = \log 2$ into the asymptotic expression 
of Theorem~\ref{thm:diag}.  
We proceed in a piecemeal way.  

To begin, note first that with $x = y = \log 2$
and $r = s = n/2$, we have 
\begin{equation}
\label{eqn:diagasym1}  
Q(x, y) = \frac{1}{4} (1 - \log 2) \log^3 2 \quad \implies \quad 
x^{-r} y^{-s} \sqrt{\frac{ye^{-y}}{sQ(x,y)}} = \frac{2}{\log^{n+1} 2} \cdot \frac{1}{\sqrt{n (1 - \log 2)}}.  
\end{equation}  
With $r = s = n/2$, we use Stirling's approximation to conclude 
\begin{equation}
\label{eqn:diagasym2}  
\frac{1}{\sqrt{2\pi}} r! s! = 
\frac{1}{\sqrt{2\pi}} \left(\left(\frac{n}{2}\right)!\right)^2 = \left(1 + O\left(\frac{1}{n}\right)\right) \frac{\sqrt{n}}{2^{n+1}} n!.    
\end{equation}
With $y = \log 2$ and $s = n/2$, the error term of Theorem~\ref{thm:multivar} is 
\begin{equation}
\label{eqn:diagasym3}  
e^{-y} + O (s^{-1/2}) = \frac{1}{2} + O \left(\frac{1}{\sqrt{n}}\right).  
\end{equation}
Multiplying equations (\ref{eqn:diagasym1})--(\ref{eqn:diagasym3}) together yields  
$$
\left[b^{\lfloor (n-1)/2 \rfloor} a^{\lceil (n-1)/2 \rceil}\right]
= 
\left( \frac{1}{2 \log 2\sqrt{  (1 - \log 2)}} + O\left(\frac{1}{\sqrt{n}}\right)  \right) \left( \frac{1}{ 2 \log 2 } \right)^n n!, 
$$
which is slightly stronger than promised.

\subsection{Proof of Theorem~\ref{thm:diag} ($n$ is odd)}  
The calculations here are similar, and we claim 
that equations (\ref{eqn:diagasym1})--(\ref{eqn:diagasym3}) still hold as long as we dampen them to include some error factor of $(1 + o(1))$, where $o(1) \to 0$ as $n\to \infty$.
When $n$ is odd, we have $r = (n+1)/2 = s + 1$, and 
it is routine to update equation (\ref{eqn:diagasym2}) to say 
\begin{equation}
\label{eqn:diagasym2'} 
\frac{1}{\sqrt{2\pi}} r! s! = 
\frac{1}{\sqrt{2\pi}} 
\left(\frac{n+1}{2}\right)! \left(\frac{n-1}{2}\right)!  
= 
(1 + o(1)) \frac{\sqrt{n}}{2^{n+1}} n!.  
\end{equation}   
Now, 
\begin{equation}
\label{eqn:Taylorxy}  
x = f^{-1}\left(\frac{s}{r}\right) = 
f^{-1} \left(1 - \frac{2}{n+1} \right) \quad \text{and} \quad 
y = f^{-1}\left(\frac{r}{s}\right) = 
f^{-1} \left(1 + \frac{2}{n-1} \right), 
\end{equation}  
and so by the continuity of $f^{-1}$, we have 
$x, y \to f^{-1}(1) = \log 2$ as $n \to \infty$.   
By continuity alone, we may 
update equation (\ref{eqn:diagasym3}) and parts  
of equation (\ref{eqn:diagasym1}) to say 
\begin{equation}
\label{eqn:diagasym3'}
e^{-y} + O (s^{-1/2}) = 
\frac{1}{2} + o(1) 
\end{equation}  
and 
\begin{equation}
\label{eqn:diagasym1'}
Q(x, y) = (1 + o(1)) \frac{1}{4} (1 - \log 2) \log^3 2  \quad \implies \quad 
\sqrt{\frac{ye^{-y}}{sQ(x,y)}} = (1 + o(1)) \frac{2}{\log 2} \cdot \frac{1}{\sqrt{n (1 - \log 2)}}.  
\end{equation}
With its large exponents, 
we need to be slightly more careful with the 
the factor $x^{-r} y^{-s}$ in equation (\ref{eqn:diagasym1}).

The function $f$ defined in equation (\ref{eqn:functionf}) is analytic, invertible, and $f'$ is nonzero.  
As such, $f^{-1}$ is analytic, and so we use  Taylor's theorem to perform a linear approximation of $f^{-1}(z)$ near $a = 1$.  
Set $d = (f^{-1})'(1)$.  
We update equation (\ref{eqn:Taylorxy}) to say   
\begin{multline}  
\label{eqn:TaylorxyUpdate}
x = 
f^{-1} \left(1 - \frac{2}{n+1} \right) 
= 
f^{-1}(1) - d \frac{2}{n+1} + \Theta\left(\frac{1}{n^2}\right) 
= 
\log 2 - d \frac{2}{n+1} + \Theta\left(\frac{1}{n^2}\right) \\
\text{and} \quad 
y = 
f^{-1} \left(1 + \frac{2}{n-1} \right)
= 
f^{-1}(1) + d \frac{2}{n-1} + \Theta\left(\frac{1}{n^2}\right)
= \log 2 + d \frac{2}{n-1} + \Theta\left(\frac{1}{n^2}\right), 
\end{multline}  
and so
$xy = \log^2 2 + \Theta(1/n^2)$.   Now, with $r = (n+1)/2 = s + 1$, 
\begin{multline}  
\label{eqn:diagasym1''}
x^r y^s = 
x (xy)^s = 
x \left( \log^2 2 + \Theta\left(\frac{1}{n^2}\right) \right)^s  
\stackrel{(\ref{eqn:TaylorxyUpdate})}{=}  
\left(  
\log^{2s+1} 2\right) 
\left(1 + \Theta\left(\frac{1}{n} \right) \right)  
\left( 1 + \Theta\left(\frac{1}{n^2}\right) \right)^{\frac{n-1}{2}}  \\
= 
(1 + o(1)) 
\log^n 2 
\quad 
\implies \quad 
x^{-r} y^{-s} = \frac{1}{\log^n 2} (1 + o(1)).  
\end{multline}  
Now, multiplying equations (\ref{eqn:diagasym2'}),   
(\ref{eqn:diagasym3'}),  
(\ref{eqn:diagasym1'}), and   
(\ref{eqn:diagasym1''})  
yields 
$$
\left[b^{\lfloor (n-1)/2 \rfloor} a^{\lceil (n-1)/2 \rceil}\right]
= 
\left( \frac{1}{2 \log 2\sqrt{  (1 - \log 2)}} + o(1)  \right) \left( \frac{1}{ 2 \log 2 } \right)^n n!, 
$$
as promised.

%%%%%%%%%%%%%%%%%%%%%%%%%%%%%%%%%%%%%%%%%%%%%%%%%%%%%%%%%%%%%%%%%%%%%%%%%%
%%%%%%%%%%%%%%%%%%%%%%%%%%%%%%%%%%%%%%%%%%%%%%%%%%%%%%%%%%%%%%%%%%%%%%%%%%
%%%%%%%%%%%%%%%%%%%%%%%%%                      %%%%%%%%%%%%%%%%%%%%%%%%%%%
%%%%%%%%%%%%%%%%%%%%%%%%%   Appendix           %%%%%%%%%%%%%%%%%%%%%%%%%%%
%%%%%%%%%%%%%%%%%%%%%%%%%                      %%%%%%%%%%%%%%%%%%%%%%%%%%%
%%%%%%%%%%%%%%%%%%%%%%%%%%%%%%%%%%%%%%%%%%%%%%%%%%%%%%%%%%%%%%%%%%%%%%%%%%
%%%%%%%%%%%%%%%%%%%%%%%%%%%%%%%%%%%%%%%%%%%%%%%%%%%%%%%%%%%%%%%%%%%%%%%%%%

\section*{Appendix: Proof of Fact~\ref{fact1}}

Recall that we wish to prove $f(t) = t^{-1} (1 - e^t) \log (1 - e^{-t})$ is a strictly decreasing bijection from $\mathbb{R}^+$ onto $\mathbb{R}^+$.  
To that end, it is straightforward to check that 
$\lim_{t \to 0^+} f(t) = + \infty$
and $\lim_{t \to +\infty} f(t) = 0$
(the numerator of the second limit tends to 1).  By continuity, we immediately conclude that $f: \mathbb{R}^+ \to \mathbb{R}^+$ is onto.

The remainder of the Appendix is reserved to proving that $f: \mathbb{R}^+ \to \mathbb{R}^+$ is strictly decreasing, and in particular, that $f'$ is negative on $\mathbb{R}^+$.    
In what follows, we make the substitution $t = \log p$, where $p \in (1, \infty)$, and we consider
the function 
$$
F(p) = f(\log p) 
= 
\frac{(1-p)\log(1-\frac{1}{p})}{\log p}
= 
\frac{(1-p) \left( \log (p-1) - \log p \right)}{\log p}  
= 
\frac{(1-p)\log (p-1)}{\log p}  
- (1 - p).   
$$
We will show that $F'$ is negative on $(1, \infty)$, and since 
$t = \log p$ is increasing in $p \in (1, \infty)$, we will infer that $f'$ is negative on $\mathbb{R}^+$.  
This will conclude our proof.

Observe that 
$$
F'(p) - 1 = 
\frac{1}{\log^2 p}  \left[ \log p \left[(-1) \log (p-1) + (1-p) \frac{1}{p-1} \right] - (1 - p) \log (1- p) \cdot \frac{1}{p} \right],   
$$
so that 
$$
F'(p) = 
- 
\frac{p (\log p) \log(p-1) + p \log p + (1- p) \log(p-1) - p \log^2 p}{p \log^2 p}.  
$$
We show that the numerator
$$
N(p) = 
p (\log p) \log(p-1) + p \log p + (1- p) \log(p-1) - p \log^2 p  
$$
is positive on $(1, \infty)$, which then implies that $F'$ is negative on $(1, \infty)$.  
For that, we first note that $N(p) \to 0^+$ as $p \to 1^+$.  As such, it suffices to show that $N'$ is positive on $(1, \infty)$.

Observe that 
\begin{eqnarray*}  
N'(p) & = & (\log p) \log(p-1) + \log (p - 1) + \frac{p}{p-1} \log p + \log p + 1 - \log (p-1) - 1 - \log^2 p - 2 \log p \\
& = & 
(\log p) \left[\frac{1}{p-1}  - \log \left(1 + \frac{1}{p-1} \right) \right].    
\end{eqnarray*}  
Note that $\log p > 0$ on $(1, \infty)$.  The second factor above is also positive, since $\log (1 + x) < x$ holds for all $x = 1/(p-1) \in \mathbb{R}$.

%%%%%%%%%%%%%%%%%%%%%%%%%%%%%%%%%%%%%%%%%%%%%%%%%%%%%%%%%%%%%%%%%%%%%%%%%%
%%%%%%%%%%%%%%%%%%%%%%%%%%%%%%%%%%%%%%%%%%%%%%%%%%%%%%%%%%%%%%%%%%%%%%%%%%
%%%%%%%%%%%%%%%%%%%%%%%%%                      %%%%%%%%%%%%%%%%%%%%%%%%%%%
%%%%%%%%%%%%%%%%%%%%%%%%%  bibliography        %%%%%%%%%%%%%%%%%%%%%%%%%%%
%%%%%%%%%%%%%%%%%%%%%%%%%                      %%%%%%%%%%%%%%%%%%%%%%%%%%%
%%%%%%%%%%%%%%%%%%%%%%%%%%%%%%%%%%%%%%%%%%%%%%%%%%%%%%%%%%%%%%%%%%%%%%%%%%
%%%%%%%%%%%%%%%%%%%%%%%%%%%%%%%%%%%%%%%%%%%%%%%%%%%%%%%%%%%%%%%%%%%%%%%%%%

\bibliographystyle{amsplain}

\end{document}